\documentclass[12pt]{article}
\usepackage{amssymb}
\usepackage{amsthm}
\usepackage{amsmath}
\usepackage{graphicx}
\usepackage{enumerate}

\DeclareMathOperator{\DM}{DM}
\DeclareMathOperator{\BDM}{\mathbf{DM}}

\newtheorem{theorem}{Theorem}[section]
\newtheorem{definition}[theorem]{Definition}
\newtheorem{lemma}[theorem]{Lemma}

\newtheorem{remark}[theorem]{Remark}
\newtheorem{example}[theorem]{Example}

\title{Kleene posets and pseudo-Kleene posets}
\author{Ivan~Chajda and Helmut~L\"anger$^1$}
\date{}
\begin{document}
\footnotetext[1]{Corresponding author, helmut.laenger@tuwien.ac.at}
\footnotetext[2]{Support of the research of the authors by the Austrian Science Fund (FWF), project I~4579-N, and the Czech Science Foundation (GA\v CR), project 20-09869L, entitled ``The many facets of orthomodularity'', as well as by \"OAD, project CZ~02/2019, entitled ``Function algebras and ordered structures related to logic and data fusion'', and, concerning the first author, by IGA, project P\v rF~2020~014, is gratefully acknowledged.}
\maketitle
\begin{abstract}
The concept of a Kleene algebra (sometimes also called Kleene lattice) was already generalized by the first author for non-distributive lattices under the name pseudo-Kleene algebra. We extend these concepts to posets and show how (pseudo-)Kleene posets can be characterized by identities and implications of assigned commutative meet-directoids. Moreover, we prove that the Dedekind-MacNeille completion of a pseudo-Kleene poset is a pseudo-Kleene algebra and that the Dedekind-MacNeille completion of a finite Kleene poset is a Kleene algebra. Further, we introduce the concept of a strict (pseudo-)Kleene poset and show that under an additional assumption a strict Kleene poset can be organized into a residuated structure. Finally, we prove by using the so-called twist construction that every poset can be embedded into a pseudo-Kleene poset in some natural way.
\end{abstract}


poset, strict (pseudo-)Kleene poset, commutative meet-directoid, Dedekind-\\
-MacNeille completion, twist construction

\section{Introduction}

Kleene algebras, or in another terminology Kleene lattices (see \cite{Ci} and \cite K), are special cases of De Morgan algebras, i.e.\ distributive lattices with an antitone involution satisfying the so-called {\em normality condition}, i.e.\ the identity $x\wedge x'\leq y\vee y'$. These algebras were also called {\em normal i-lattices} (see \cite K) or quasi-Boolean algebras (by A.~Bialynicki and H.~Rasiova). They are important models in the field of logic since they generalize Boolean algebras, \L ukasiewicz algebras and Post algebras. The name ``Kleene algebra'' was introduced by R.~Cignoli (\cite{Ci}). The case when the underlying lattice need not be distributive was treated by the first author in \cite{Ch} under the name pseudo-Kleene algebras.

In some propsitional logics the identification of disjunction with lattice join $\vee$ turns out to be problematic. For example, the logic of quantum mechanics was originally modeled by orthomodular lattices (which are special pseudo-Kleene algebras) and later on by orthomodular posets in which the existence of the join $x\vee y$ of two elements $x$ and $y$ is guaranteed only in the case when these elements are orthogonal to each other which means that $x\leq y'$ (or, equivalently, $y\leq x'$). Similar problems may occur also in other models of non-classical logics. Hence the question arises whether results obtained for Kleene algebras or pseudo-Kleene algebras can be generalized to posets with an antitone involution satisfying some condition analogous to normality. In this paper we solve this problem by investigating so-called Kleene posets, pseudo-Kleene posets and strong pseudo-Kleene posets. We believe that these may be successfully applied in the algebraic axiomatization of several non-classical logics. Moreover, we introduce some kind of residuation which may be applied in fuzzy logic.

For the reader's convenience, we recall several concepts concerning posets.

Let $\mathbf P=(P,\leq)$ be a poset, $a,b\in P$ and $A,B\subseteq P$. We write $a\parallel b$ if $a$ and $b$ are incomparable and we extend $\leq$ to subsets by defining
\[
A\leq B\text{ if and only if }x\leq y\text{ for all }x\in A\text{ and }y\in B.
\]
Instead of $\{a\}\leq B$ and $A\leq\{b\}$ we also write $a\leq B$ and $A\leq b$, respectively. Analogous notations are used for the reverse order $\geq$. Moreover, we define
\begin{align*}
L(A) & :=\{x\in P\mid x\leq A\}, \\
U(A) & :=\{x\in P\mid A\leq x\}.
\end{align*}
Instead of $L(A\cup B)$, $L(\{a\}\cup B)$, $L(A\cup\{b\})$ and $L(\{a,b\})$ we also write $L(A,B)$, $L(a,B)$, $L(A,b)$ and $L(a,b)$, respectively. Analogous notations are used for $U$. Instead of $L(U(A))$ we also write $LU(A)$. Analogously, we proceed in similar cases. Sometimes we identify singletons with their unique element, so we often write $L(a,b)=0$ and $U(a,b)=1$ instead of $L(a,b)=\{0\}$ and $U(a,b)=\{1\}$, respectively. The {\em poset} $\mathbf P$ is called {\em downward directed} if $L(x,y)\neq\emptyset$ for all $x,y\in P$. The {\em poset} $\mathbf P$ is called {\em bounded} if it has a least element $0$ and a greatest element $1$. This fact will be expressed by notation $(P,\leq,0,1)$. The {\em poset} $\mathbf P$ is called {\em distributive} if it satisfies one of the following equivalent identities:
\begin{align*}
 L(U(x,y),z) & \approx LU(L(x,z),L(y,z)), \\
UL(U(x,y),z) & \approx U(L(x,z),L(y,z)), \\
 U(L(x,y),z) & \approx UL(U(x,z),U(y,z)), \\
LU(L(x,y),z) & \approx L(U(x,z),U(y,z)).
\end{align*}
In fact, the inclusions
\begin{align*}
LU(L(x,z),L(y,z)) & \subseteq L(U(x,y),z), \\
UL(U(x,z),U(y,z)) & \subseteq U(L(x,y),z)
\end{align*}
hold in every poset. Hence, to check distributivity, we need only to confirm one of the converse inclusions. A unary operation $'$ on $P$ is called
\begin{itemize}
\item {\em antitone} if, for all $x,y\in P$, $x\leq y$ implies $y'\leq x'$,
\item an {\em involution} if it satisfies the identity $x''\approx x$.
\end{itemize}
For $A\subseteq P$ we define $A':=\{x'\mid x\in A\}$. If the poset is bounded and distributive, we can prove the following property of an antitone involution.

\begin{lemma}
Let $(P,\leq,{}',0,1)$ be a bounded distributive poset with an antitone involution and $a,b\in P$ with $a\leq b$ and $L(b,a')=\{0\}$. Then the following hold:
\begin{align*}
L(a,a') & =L(b,b')=\{0\}, \\
U(a,a') & =U(b,b')=\{1\}.
\end{align*}
\end{lemma}

\begin{proof}
We have
\begin{align*}
L(a,a') & =LUL(a,a')=LU(L(a,a'),0)=LU(L(a,a'),L(b,a'))=L(U(a,b),a')= \\
        & =L(U(b),a')=L(b,a')=\{0\}, \\
L(b,b') & =LUL(b',b)=LU(0,L(b',b))=LU(L(a',b),L(b',b))=L(U(a',b'),b)= \\
        & =L(U(a'),b)=L(a',b)=\{0\}, \\
U(a,a') & =(L(a',a))'=\{0\}'=\{1\}, \\
U(b,b') & =(L(b',b))'=\{0\}'=\{1\}.
\end{align*}
\end{proof}

\section{Kleene posets and pseudo-Kleene posets}

Now we define our main concepts.

\begin{definition}
A {\em pseudo-Kleene poset} is a poset $\mathbf P=(P,\leq,{}')$ with an antitone involution satisfying
\begin{enumerate}
\item[{\rm(K)}] $L(x,x')\leq U(y,y')$ for all $x,y\in P$.
\end{enumerate}
An element $a$ of $P$ is called a {\em fixed point} of $\mathbf P$ if $a'=a$. By a {\em Kleene poset} we mean a distributive pseudo-Kleene poset.
\end{definition}

\begin{lemma}
Let $\mathbf P=(P,\leq,{}')$ be a pseudo-Kleene poset. Then $\mathbf P$ has at most one fixed point.
\end{lemma}

\begin{proof}
If $a$ and $b$ are fixed points of $\mathbf P$ then
\begin{align*}
L(a) & =L(a,a)=L(a,a')\leq U(b,b')=U(b,b)=U(b), \\
L(b) & =L(b,b)=L(b,b')\leq U(a,a')=U(a,a)=U(a)
\end{align*}
and hence $a\leq b$ and $b\leq a$, i.e.\ $a=b$.
\end{proof}

\begin{example}\label{ex2}
The poset visualized in Figure~1 is a Kleene poset which is not a lattice:
\vspace*{-2mm}
\begin{center}
\setlength{\unitlength}{7mm}
\begin{picture}(4,8)
\put(2,1){\circle*{.3}}
\put(1,3){\circle*{.3}}
\put(3,3){\circle*{.3}}
\put(1,5){\circle*{.3}}
\put(3,5){\circle*{.3}}
\put(2,7){\circle*{.3}}
\put(1,3){\line(0,1)2}
\put(1,3){\line(1,1)2}
\put(1,3){\line(1,-2)1}
\put(3,3){\line(-1,-2)1}
\put(3,3){\line(-1,1)2}
\put(3,3){\line(0,1)2}
\put(2,7){\line(-1,-2)1}
\put(2,7){\line(1,-2)1}
\put(1.85,.25){$0$}
\put(.3,2.85){$a$}
\put(3.4,2.85){$b$}
\put(.3,4.85){$b'$}
\put(3.4,4.85){$a'$}
\put(1.3,7.4){$1=0'$}
\put(1.2,-.75){{\rm Fig.\ 1}}
\end{picture}
\end{center}
\vspace*{4mm}
\end{example}

\begin{example}
The poset $(P,\leq,{}')$ visualized in Figure~2 is a pseudo-Kleene poset which is neither a lattice nor a Kleene poset since
\[
L(U(a,c),b)=L(1,b)=L(b)\neq\{0\}=L(P)=LU(0)=LU(0,0)=LU(L(a,b),L(c,b)).
\]
\vspace*{-2mm}
\begin{center}
\setlength{\unitlength}{7mm}
\begin{picture}(8,8)
\put(4,1){\circle*{.3}}
\put(3,3){\circle*{.3}}
\put(5,3){\circle*{.3}}
\put(3,5){\circle*{.3}}
\put(5,5){\circle*{.3}}
\put(4,7){\circle*{.3}}
\put(1,4){\circle*{.3}}
\put(7,4){\circle*{.3}}
\put(3,3){\line(0,1)2}
\put(3,3){\line(1,1)2}
\put(3,3){\line(1,-2)1}
\put(5,3){\line(-1,-2)1}
\put(5,3){\line(-1,1)2}
\put(5,3){\line(0,1)2}
\put(4,7){\line(-1,-2)1}
\put(4,7){\line(1,-2)1}
\put(4,1){\line(-1,1)3}
\put(4,1){\line(1,1)3}
\put(4,7){\line(-1,-1)3}
\put(4,7){\line(1,-1)3}
\put(3.85,.25){$0$}
\put(2.3,2.85){$a$}
\put(5.4,2.85){$b$}
\put(.3,3.85){$c$}
\put(7.4,3.85){$c'$}
\put(2.3,4.85){$b'$}
\put(5.4,4.85){$a'$}
\put(3.3,7.4){$1=0'$}
\put(3.2,-.75){{\rm Fig.\ 2}}
\end{picture}
\end{center}
\vspace*{4mm}
\end{example}

Recall that a {\em pseudo-Kleene algebra} (see \cite{Ch}) is a lattice $(L,\vee,\wedge,{}')$ with an antitone involution satisfying the identity
\[
x\wedge x'\leq y\vee y'\text{ for all }x,y\in L.
\]
A {\em Kleene algebra} (see \cite K) is a distributive pseudo-Kleene algebra. Observe that a lattice with an antitone involution is a (pseudo-)Kleene poset if and only if it is a (pseudo-)Kleene algebra. Hence, our concepts defined above are appropriate generalizations of pseudo-Kleene algebras and Kleene algebras as will be shown in the next section.

\begin{example}
The lattice visualized in Figure~3 is a pseudo-Kleene algebra which is not a Kleene algebra since the lattice is not distributive:
\vspace*{-2mm}
\begin{center}
\setlength{\unitlength}{7mm}
\begin{picture}(4,10)
\put(2,1){\circle*{.3}}
\put(1,3){\circle*{.3}}
\put(3,3){\circle*{.3}}
\put(1,5){\circle*{.3}}
\put(3,5){\circle*{.3}}
\put(1,7){\circle*{.3}}
\put(3,7){\circle*{.3}}
\put(2,9){\circle*{.3}}
\put(1,3){\line(0,1)4}
\put(1,3){\line(1,2)2}
\put(1,3){\line(1,-2)1}
\put(3,3){\line(-1,-2)1}
\put(3,3){\line(0,1)4}
\put(2,9){\line(-1,-2)1}
\put(2,9){\line(1,-2)1}
\put(1.85,.25){$0$}
\put(.3,2.85){$a$}
\put(3.4,2.85){$b$}
\put(.3,4.85){$c$}
\put(3.4,4.85){$c'$}
\put(.3,6.85){$b'$}
\put(3.4,6.85){$a'$}
\put(1.3,9.4){$1=0'$}
\put(1.2,-.75){{\rm Fig.\ 3}}
\end{picture}
\end{center}
\vspace*{4mm}
\end{example}

\begin{example}
The lattice visualized in Figure~4 is a Kleene algebra:
\vspace*{-2mm}
\begin{center}
\setlength{\unitlength}{7mm}
\begin{picture}(4,10)
\put(2,1){\circle*{.3}}
\put(2,3){\circle*{.3}}
\put(1,5){\circle*{.3}}
\put(3,5){\circle*{.3}}
\put(2,7){\circle*{.3}}
\put(2,9){\circle*{.3}}
\put(2,3){\line(0,-1)2}
\put(2,3){\line(-1,2)1}
\put(2,3){\line(1,2)1}
\put(2,7){\line(-1,-2)1}
\put(2,7){\line(1,-2)1}
\put(2,7){\line(0,1)2}
\put(1.85,.25){$0$}
\put(2.4,2.85){$a$}
\put(.3,4.85){$b$}
\put(3.4,4.85){$b'$}
\put(2.4,6.85){$a'$}
\put(1.3,9.4){$1=0'$}
\put(1.2,-.75){{\rm Fig.\ 4}}
\end{picture}
\end{center}
\vspace*{4mm}
\end{example}

\section{Dedekind-MacNeille completion}

In what follows we investigate the question for which posets $\mathbf P$ with an antitone involution their Dedekind-MacNeille completion $\BDM(\mathbf P)$ is either a pseudo-Kleene algebra or a Kleene algebra.

Let $\mathbf P=(P,\leq,{}')$ be a poset with an antitone involution. Define
\begin{align*}
 \DM(\mathbf P) & :=\{L(A)\mid A\subseteq P\}, \\
            A^* & :=L(A')\text{ for all }A\in\DM(\mathbf P), \\
\BDM(\mathbf P) & :=(\DM(\mathbf P),\subseteq,^*)
\end{align*}
Then $\BDM(\mathbf P)$ is a complete lattice with an antitone involution, called the {\em Dedekind-MacNeille completion} of $\mathbf P$. That $^*$ is an antitone involution on $(\DM(\mathbf P),\subseteq)$ can be seen as follows. Let $A,B\in\DM(\mathbf P)$. If $A\subseteq B$ then $A'\subseteq B'$ and hence $B^*=L(B')\subseteq L(A')=A^*$. Moreover, $A^{**}=L((L(A'))')=LU(A)=A$. We have
\begin{align*}
 (L(A))^* & =L((L(A))')=LU(A')\text{ for all }A\subseteq P, \\
  A\vee B & =LU(A,B)\text{ for all }A,B\in\DM(\mathbf P), \\
A\wedge B & =A\cap B\text{ for all }A,B\in\DM(\mathbf P).
\end{align*}

\begin{theorem}\label{th6}
Let $\mathbf P=(P,\leq,{}')$ be a poset with an antitone involution. Then $\BDM(\mathbf P)$ is a pseudo-Kleene algebra if and only if $\mathbf P$ is a pseudo-Kleene poset.
\end{theorem}

\begin{proof}
Assume that $\mathbf P$ is a pseudo-Kleene poset. Then the following are equivalent:
\begin{align*}
                   & \BDM(\mathbf P)\text{ is a pseudo-Kleene algebra}, \\
       C\wedge C^* & \subseteq D\vee D^*\text{ for all }C,D\in\DM(\mathbf P), \\
L(A)\wedge(L(A))^* & \subseteq L(B)\vee(L(B))^*\text{ for all }A,B\subseteq P, \\
 L(A)\cap L(U(A')) & \subseteq LU(L(B),LU(B'))\text{ for all }A,B\subseteq P, \\
        L(A,U(A')) & \subseteq L(UL(B)\cap ULU(B'))\text{ for all }A,B\subseteq P, \\
        L(A,U(A')) & \subseteq L(UL(B)\cap U(B'))\text{ for all }A,B\subseteq P, \\
        L(A,U(A')) & \subseteq LU(L(B),B')\text{ for all }A,B\subseteq P, \\
        L(A,U(A')) & \leq U(L(B),B')\text{ for all }A,B\subseteq P.
\end{align*}
Now let $A,B$ be fixed subsets of $P$. Then
\begin{align*}
L(A,U(A')) & =L(A)\cap LU(A')=\bigcup_{x\in L(A)}L(x)\cap\bigcap_{y\in U(A')}L(y)= \\
& =\bigcup_{x\in L(A)}(L(x)\cap\bigcap_{y\in U(A')}L(y))\subseteq\bigcup_{x\in L(A)}(L(x)\cap L(x'))=\bigcup_{x\in L(A)}L(x,x'), \\
U(L(B),B') & =UL(B)\cap U(B')=\bigcap_{x\in L(B)}U(x)\cap\bigcup_{y\in U(B')}U(y)= \\
& =\bigcup_{y\in U(B')}(\bigcap_{x\in L(B)}U(x)\cap U(y))\subseteq\bigcup_{y\in U(B')}(U(y')\cap U(y))=\bigcup_{y\in U(B')}U(y,y')
\end{align*}
and
\[
\bigcup_{x\in L(A)}L(x,x')\leq\bigcup_{y\in U(B')}U(y,y').
\]
Hence $\BDM(\mathbf P)$ is a pseudo-Kleene algebra provided $\mathbf P$ is a pseudo-Kleene poset. The converse is evident.
\end{proof}

\begin{theorem}
If $\mathbf P=(P,\leq,{}')$ is a finite poset with an antitone involution then $\BDM(\mathbf P)$ is a Kleene algebra if and only if $\mathbf P$ is a Kleene poset.
\end{theorem}

\begin{proof}
According to a result by M.~Ern\'e (\cite E), the Dedekind-MacNeille completion of a finite distributive poset is a distributive lattice. Hence, by Theorem~\ref{th6} the Dedekind-MacNeille completion of a finite Kleene poset is a Kleene algebra. The converse is evident.
\end{proof}

The Dedekind-MacNeille completion of the Kleene poset from Example~\ref{ex2} is visualized in Figure~5:
\vspace*{-2mm}
\begin{center}
\setlength{\unitlength}{7mm}
\begin{picture}(4,10)
\put(2,1){\circle*{.3}}
\put(1,3){\circle*{.3}}
\put(3,3){\circle*{.3}}
\put(2,5){\circle*{.3}}
\put(1,7){\circle*{.3}}
\put(3,7){\circle*{.3}}
\put(2,9){\circle*{.3}}
\put(1,3){\line(1,2)2}
\put(1,3){\line(1,-2)1}
\put(3,3){\line(-1,2)2}
\put(3,3){\line(-1,-2)1}
\put(2,9){\line(-1,-2)1}
\put(2,9){\line(1,-2)1}
\put(1.85,.25){$0$}
\put(.3,2.85){$a$}
\put(3.4,2.85){$b$}
\put(2.4,4.85){$c=c'$}
\put(.3,6.85){$b'$}
\put(3.4,6.85){$a'$}
\put(1.3,9.4){$1=0'$}
\put(1.2,-.75){{\rm Fig.\ 5}}
\end{picture}
\end{center}
\vspace*{4mm}

\section{A representation by commutative meet-directoids}

Now we recall the concept of a commutative meet-directoid. We will use it for the characterization of pseudo-Kleene posets, Kleene posets, strong pseudo-Kleene posets and strict pseudo-Kleene posets. The advantage of this approach is that we characterize properties of posets by means of identities and quasiidentities of algebras. Hence, one can use algebraic tools for their investigation.

A {\em commutative meet-directoid} (see \cite{CL11} and \cite{JQ}) is a groupoid $\mathbf D=(D,\sqcap)$ satisfying the following identities:
\begin{align*}
                   x\sqcap x & \approx x\text{ (idempotency)}, \\
                   x\sqcap y & \approx y\sqcap x\text{ (commutativity)}, \\
(x\sqcap(y\sqcap z))\sqcap z & \approx x\sqcap(y\sqcap z)\text{ (weak associativity)}.
\end{align*}
If $\mathbf P=(P,\leq)$ is a downward directed poset, if we define $x\sqcap y:=x\wedge y$ for comparable $x,y\in P$ and if we put for $x\sqcap y=y\sqcap x$ an arbitrary element of $L(x,y)$ if $x,y\in P$ are incomparable, then $\mathbb D(\mathbf P):=(P,\sqcap)$ is a commutative meet-directoid which is called a {\em meet-directoid assigned} to $\mathbf P$. Conversely, if $\mathbf D=(D,\sqcap)$ is a commutative meet-directoid and we define
\[
x\leq y\text{ if and only if }x\sqcap y=x
\]
for all $x,y\in D$ then $\mathbb P(\mathbf D):=(D,\leq)$ is a downward directed poset, the so-called {\em poset induced} by $\mathbf D$. Though the assignment $\mathbf P\mapsto\mathbb D(\mathbf P)$ is not unique, we have $\mathbb P(\mathbb D(\mathbf P))=\mathbf P$ for every downward directed poset $\mathbf P$. Sometimes we consider posets and commutative meet-directoids together with a unary operation. Let $(D,\sqcap,{}')$ be a commutative meet-directoid $(D,\sqcap,{}')$ with an antitone involution. W e define
\[
x\sqcup y:=(x'\sqcap y')'\text{ for all }x,y\in D.
\]
Then $\sqcup$ is also idempotent, commutative and weakly associative and we have for all $x,y\in D$
\begin{align*}
x\sqcup y & =x\vee y\text{ if }x,y\text{ are }comparable, \\
x\sqcup y & =y\sqcup x\in U(x,y)\text{ if }x\parallel y, \\
x\sqcap y & =x\text{ if and only if }x\sqcup y=y, \\
     L(x) & =\{z\sqcap x\mid z\in P\}, \\
     U(x) & =\{z\sqcup x\mid z\in P\}, \\
   L(x,y) & =\{(z\sqcap x)\sqcap(z\sqcap y)\mid z\in P\}, \\
   U(x,y) & =\{(z\sqcup x)\sqcap(z\sqcup y)\mid z\in P\}.
\end{align*}

Posets with an antitone involution can be characterized in the language of commutative meet-directoids by identities as follows.

\begin{lemma}\label{lem1}
Let $\mathbf P=(P,\leq,{}')$ be a downward directed poset with a unary operation and $\mathbb D(\mathbf P)$ an assigned meet-directoid. Then $\mathbf P$ is a poset with an antitone involution if and only if $\mathbb D(\mathbf P)$ satisfies the identities
\begin{enumerate}[{\rm(1)}]
\item $x''\approx x$,
\item $(x\sqcap y)'\sqcap y'\approx y'$.
\end{enumerate}
\end{lemma}

\begin{proof}
Condition (1) is evident by definition. Further, the following are equivalent:
\begin{align*}
& (2), \\
& x'\sqcap y'=y'\text{ for all }x,y\in P\text{ with }x\leq y, \\
& y'\leq x'\text{ for all }x,y\in P\text{ with }x\leq y, \\
& '\text{ is antitone}.
\end{align*}
\end{proof}

Now we characterize pseudo-Kleene posets by identities of an assigned commutative meet-directoid.

\begin{theorem}\label{th1}
Let $\mathbf P=(P,\leq,{}')$ be a downward directed poset with a unary operation and $\mathbb D(\mathbf P)$ an assigned meet-directoid. Then $\mathbf P$ is a pseudo-Kleene poset if and only if $\mathbb D(\mathbf P)$ satisfies identities {\rm(1)} -- {\rm(3)}:
\begin{enumerate}
\item[{\rm(3)}] $(z\sqcap x)\sqcap(z\sqcap x')\leq(w\sqcup y)\sqcup(w\sqcup y')$.
\end{enumerate}
\end{theorem}

\begin{proof}
It is easy to check that (3) is equivalent to $L(x,x')\leq U(y,y')$ for all $x,y\in P$. Applying Lemma~\ref{lem1} completes the proof.
\end{proof}

In order to characterize Kleene posets in a similar way we need to capture distributivity of posets in the language of commutative meet-directoids.

\begin{theorem}\label{th3}
Let $\mathbf P=(P,\leq,{}')$ be a downward directed poset with a unary operation and $\mathbb D(\mathbf P)$ an assigned meet-directoid. Then $\mathbf P$ is a Kleene poset if and only if $\mathbb D(\mathbf P)$ satisfies identities {\rm(1)} -- {\rm(3)} and implication {\rm(4)}:
\begin{enumerate}
\item[{\rm(4)}] $w\sqcap((t\sqcup x)\sqcup(t\sqcup y))=w\sqcap z=w$ and $s\sqcup((t\sqcap x)\sqcap(t\sqcap z))=s\sqcup((t\sqcap y)\sqcap(t\sqcap z))=s$ for all $t\in P$ imply $w\leq s$.
\end{enumerate}
\end{theorem}

\begin{proof}
Since
\begin{align*}
& U(x,y)=\{(t\sqcup x)\sqcup(t\sqcup y)\mid t\in P\}, \\
& w\sqcap u=w\text{ is equivalent to }w\in L(u),
\end{align*}
$w\sqcap((t\sqcup x)\sqcup(t\sqcup y))=w\sqcap z=w$ is equivalent to $w\in L(U(x,y),z)$. Further, since
\begin{align*}
& L(x,z)=\{(t\sqcap x)\sqcap(t\sqcap z)\mid t\in P\}, \\
& L(y,z)=\{(t\sqcap y)\sqcap(t\sqcap z)\mid t\in P\}, \\
& s\sqcup u=s\text{ is equivalent to }s\in U(u),
\end{align*}
$s\sqcup((t\sqcap x)\sqcap(t\sqcap z))=s\sqcup((t\sqcap y)\sqcap(t\sqcap z))=s$ is equivalent to $s\in U(L(x,z),L(y,z))$. Hence the following are equivalent:
\begin{align*}
& (4), \\
& w\in L(U(x,y),z)\text{ and }s\in U(L(x,z),L(y,z))\text{ imply }w\leq s, \\
& L(U(x,y),z)\subseteq LU(L(x,z),L(y,z)), \\
& \mathbf P\text{ is distributive}.
\end{align*}
Applying Theorem~\ref{th1} completes the proof.
\end{proof}

Let us note that the class of all directoids assigned to downward directed pseudo-Kleene posets forms a variety due to Theorem~\ref{th1}. As shown in \cite{CL11} and \cite{CKL}, every variety of directoids with an antitone involution is congruence distributive.

\begin{definition}
A {\em strong pseudo-Kleene poset} is a poset $(P,\leq,{}')$ with an antitone involution satisfying
\begin{enumerate}
\item[{\rm(S)}] $x\parallel y$ implies $L(x,x')=L(y,y')$.
\end{enumerate}
\end{definition}

\begin{lemma}\label{lem2}
Let $\mathbf P=(P,\leq,{}')$ be a strong pseudo-Kleene poset. Then $\mathbf P$ is a pseudo-Kleene poset.
\end{lemma}

\begin{proof}
Let $a,b\in P$. \\
If $a\leq b$ then $L(a,a')\leq a\leq b\leq U(b,b')$. \\
If $b\leq a$ then $a'\leq b'$ and hence $L(a,a')\leq a'\leq b'\leq U(b,b')$. \\
If $a\parallel b$ then $L(a,a')=L(b,b')\leq U(b,b')$.
\end{proof}

\begin{example}\label{ex1}
The poset visualized in Figure~6 is a strong pseudo-Kleene poset which is not a lattice and hence not a pseudo-Kleene algebra:
\vspace*{-2mm}
\begin{center}
\setlength{\unitlength}{7mm}
\begin{picture}(8,16)
\put(4,1){\circle*{.3}}
\put(4,3){\circle*{.3}}
\put(4,5){\circle*{.3}}
\put(1,7){\circle*{.3}}
\put(3,7){\circle*{.3}}
\put(5,7){\circle*{.3}}
\put(7,7){\circle*{.3}}
\put(1,9){\circle*{.3}}
\put(3,9){\circle*{.3}}
\put(5,9){\circle*{.3}}
\put(7,9){\circle*{.3}}
\put(4,11){\circle*{.3}}
\put(4,13){\circle*{.3}}
\put(4,15){\circle*{.3}}
\put(4,5){\line(0,-1)4}
\put(4,5){\line(-3,2)3}
\put(4,5){\line(-1,2)1}
\put(4,5){\line(1,2)1}
\put(4,5){\line(3,2)3}
\put(1,7){\line(0,1)2}
\put(1,7){\line(1,1)2}
\put(3,7){\line(-1,1)2}
\put(3,7){\line(0,1)2}
\put(5,7){\line(0,1)2}
\put(5,7){\line(1,1)2}
\put(7,7){\line(-1,1)2}
\put(7,7){\line(0,1)2}
\put(4,11){\line(-3,-2)3}
\put(4,11){\line(-1,-2)1}
\put(4,11){\line(1,-2)1}
\put(4,11){\line(3,-2)3}
\put(4,11){\line(0,1)4}
\put(3.85,.25){$0$}
\put(4.4,2.85){$a$}
\put(4.4,4.85){$b$}
\put(.3,6.85){$c$}
\put(3.4,6.85){$d$}
\put(4.3,6.85){$e$}
\put(7.4,6.85){$f$}
\put(.3,8.85){$f'$}
\put(3.4,8.85){$e'$}
\put(4.3,8.85){$d'$}
\put(7.4,8.85){$c'$}
\put(4.4,10.85){$b'$}
\put(4.4,12.85){$a'$}
\put(3.3,15.4){$1=0'$}
\put(3.2,-.75){{\rm Fig.\ 6}}
\end{picture}
\end{center}
\vspace*{4mm}
\end{example}

We are going to determine the class of directoids assigned to strong pseudo-Kleene posets.

\begin{theorem}\label{th2}
Let $\mathbf P=(P,\leq,{}')$ be a downward directed poset with a unary operation and $\mathbb D(\mathbf P)$ an assigned meet-directoid. Then $\mathbf P$ is a strong pseudo-Kleene poset if and only if $\mathbb D(\mathbf P)$ satisfies identities {\rm(1)} and {\rm(2)} and implication {\rm(5)}:
\begin{enumerate}
\item[{\rm(5)}] $x\neq x\sqcap y\neq y$ and $x\sqcap z=x'\sqcap z=z$ imply $y\sqcap z=y'\sqcap z=z$.
\end{enumerate}
\end{theorem}

\begin{proof}
Since
\begin{align*}
& x\sqcap y=x\text{ is equivalent to }x\leq y, \\
& x\sqcap y=y\text{ is equivalent to }y\leq x
\end{align*}
we have that $x\neq x\sqcap y\neq y$ is equivalent to $x\parallel y$. Further, since
\begin{align*}
& x\sqcap z=x'\sqcap z=z\text{ is equivalent to }z\in L(x,x'), \\
& y\sqcap z=y'\sqcap z=z\text{ is equivalent to }z\in L(y,y'),
\end{align*}
the following are equivalent:
\begin{align*}
& (5), \\
& x\parallel y\text{ and }z\in L(x,x')\text{ imply }z\in L(y,y'), \\
& x\parallel y\text{ implies }L(x,x')\subseteq L(y,y'), \\
& x\parallel y\text{ implies }L(x,x')=L(y,y').
\end{align*}
Applying Lemma~\ref{lem1} completes the proof.
\end{proof}

The following concept will be used in the sequel.

\begin{definition}
A {\em strict pseudo-Kleene poset} is a bounded poset $(P,\leq,{}')$ with an antitone involution satisfying
\[
x,y\neq0,1\text{ implies }L(x,x')=L(y,y').
\]
A {\em strict Kleene poset} is a distributive strict pseudo-Kleene poset.
\end{definition}

Obviously, every strict pseudo-Kleene poset is a strong pseudo-Kleene poset and hence a pseudo-Kleene poset according to Lemma~\ref{lem2}, but not conversely (see Example~\ref{ex1} where $L(a,a')=a\neq b=L(b,b')$).

Of course, every Boolean poset, i.e.\ every bounded distributive poset where the antitone involution which is a complementation (i.e.\ $L(x,x')\approx\{0\}$ and $U(x,x')=\{1\}$), is a strict Kleene poset. In the next example we show a strict Kleene poset which is not Boolean.

\begin{example}
The poset visualized in Figure~7 is a strict Kleene poset which is not a lattice and hence not a Kleene algebra:
\vspace*{-2mm}
\begin{center}
\setlength{\unitlength}{7mm}
\begin{picture}(8,14)
\put(4,1){\circle*{.3}}
\put(4,3){\circle*{.3}}
\put(1,5){\circle*{.3}}
\put(3,5){\circle*{.3}}
\put(5,5){\circle*{.3}}
\put(7,5){\circle*{.3}}
\put(1,7){\circle*{.3}}
\put(7,7){\circle*{.3}}
\put(1,9){\circle*{.3}}
\put(3,9){\circle*{.3}}
\put(5,9){\circle*{.3}}
\put(7,9){\circle*{.3}}
\put(4,11){\circle*{.3}}
\put(4,13){\circle*{.3}}
\put(4,3){\line(0,-1)2}
\put(4,3){\line(-3,2)3}
\put(4,3){\line(-1,2)1}
\put(4,3){\line(1,2)1}
\put(4,3){\line(3,2)3}
\put(4,11){\line(-3,-2)3}
\put(4,11){\line(-1,-2)1}
\put(4,11){\line(1,-2)1}
\put(4,11){\line(3,-2)3}
\put(1,5){\line(0,1)4}
\put(7,5){\line(0,1)4}
\put(1,7){\line(1,1)2}
\put(1,5){\line(1,1)4}
\put(3,5){\line(1,1)4}
\put(5,5){\line(1,1)2}
\put(3,5){\line(-1,1)2}
\put(5,5){\line(-1,1)4}
\put(7,5){\line(-1,1)4}
\put(7,7){\line(-1,1)2}
\put(4,11){\line(0,1)2}
\put(3.85,.25){$0$}
\put(4.4,2.85){$a$}
\put(.3,4.85){$b$}
\put(2.3,4.85){$c$}
\put(5.4,4.85){$d$}
\put(7.4,4.85){$e$}
\put(.3,6.85){$f$}
\put(7.4,6.85){$f'$}
\put(2.3,8.85){$d'$}
\put(.3,8.85){$e'$}
\put(5.4,8.85){$c'$}
\put(7.4,8.85){$b'$}
\put(4.4,10.85){$a'$}
\put(3.3,13.4){$1=0'$}
\put(3.2,-.75){{\rm Fig.\ 7}}
\end{picture}
\end{center}
\vspace*{4mm}
\end{example}

Analogously as above, also strict pseudo-Kleene posets and strict Kleene posets can be characterized by means of properties of assigned meet-directoids.

\begin{theorem}
Let $\mathbf P=(P,\leq,{}')$ be a downward directed poset with a unary operation and $\mathbb D(\mathbf P)$ an assigned meet-directoid. Then the following hold:
\begin{enumerate}[{\rm(i)}]
\item $\mathbf P$ is a strict pseudo-Kleene poset if and only if $\mathbb D(\mathbf P)$ satisfies identities {\rm(1)} and {\rm(2)} and implication {\rm(6)}:
\begin{enumerate}
\item[{\rm(6)}] $x,y\notin\{0,1\}$ and $x\sqcap z=x'\sqcap z=z$ imply $y\sqcap z=y'\sqcap z=z$.
\end{enumerate}
\item $\mathbf P$ is a strict Kleene poset if and only if $\mathbb D(\mathbf P)$ satisfies identities {\rm(1)} and {\rm(2)} and implications {\rm(4)} and {\rm(6)}.
\end{enumerate}
\end{theorem}

\begin{proof}
\
\begin{enumerate}[(i)]
\item The following are equivalent:
\begin{align*}
& (6), \\
& x,y\notin\{0,1\}\text{ and }z\in L(x,x')\text{ imply }z\in L(y,y'), \\
& x,y\notin\{0,1\}\text{ implies }L(x,x')\subseteq L(y,y'), \\
& x,y\notin\{0,1\}\text{ implies }L(x,x')=L(y,y').
\end{align*}
Applying Lemma~\ref{lem1} completes the proof.
\item This follows from (i) and the proof of Theorem~\ref{th3}.
\end{enumerate}
\end{proof}

\section{Residuation in Kleene posets}

In the study of non-classical logics we prefer to have a logical connective implication which enables deduction, i.e.\ to derive new propositions from given ones. The question arises how to define implication in the logic based on pseudo-Kleene posets and Kleene posets. Usually implication is considered to be well-behaved if it is related with conjunction by means of adjointness.

In the following we will study residuation in strict pseudo-Kleene posets. Contrary to the case of lattices, $\odot$ and $\rightarrow$ cannot be binary operations, they are only operators, i.e.\ mappings from $P^2$ to $2^P$. We extend $\odot$ to $(2^P)^2$ by defining
\[
A\odot B:=\bigcap_{x\in A,y\in B}(x\odot y)
\]
for all $A,B\subseteq P$.

\begin{definition}
A {\em Kleene residuated poset} is an ordered six-tuple $(P,\leq,\odot,\rightarrow,0,1)$ \\
where $(P,\leq,0,1)$ is a bounded strict pseudo-Kleene poset and $\odot$ and $\rightarrow$ are mappings from $P^2$ to $2^P$ satisfying the following conditions for all $x,y,z\in P$:
\begin{itemize}
\item $x\odot y\approx y\odot x$,
\item $x\odot1\approx1\odot x\approx L(x)$,
\item $(x\odot y)\odot z\approx x\odot(y\odot z)$,
\item $x\odot y\leq z$ if and only if $x\leq y\rightarrow z$ {\rm(}adjointness{\rm)}.
\end{itemize}
\end{definition}

Let $(P,\leq,{}')$ be a poset with an antitone involution. Define mappings $\odot$ and $\rightarrow$ from $P^2$ to $2^P$ as follows:
\begin{enumerate}
\item[(R)] $\quad x\odot y:=\left\{
\begin{array}{ll}
0      & \text{if }x\leq y', \\
L(x,y) & \text{otherwise}
\end{array}
\right.
\quad\quad\quad x\rightarrow y:=\left\{
\begin{array}{ll}
1       & \text{if }x\leq y, \\
U(x',y) & \text{otherwise}
\end{array}
\right.$
\end{enumerate}

\begin{theorem}\label{th4}
Let $(P,\leq,{}',0,1)$ be a bounded strict Kleene poset satisfying
\begin{enumerate}
\item[{\rm(7)}] $L(x,y)\neq0$ for all $x,y\in P\setminus\{0\}$
\end{enumerate}
and let $\odot$ and $\rightarrow$ be defined by {\rm(R)}. Then $(P,\leq,\odot,\rightarrow,0,1)$ is a Kleene residuated poset.
\end{theorem}

\begin{proof}
Let $a,b,c\in P$. One can easily see that $x\odot0\approx0\odot x\approx0$. Since $0\in L(x,y)$ for all $x,y\in P$ and therefore $0\in x\odot y$ for all $x,y\in P$, we have $0\in a\odot b$ and $0\in b\odot c$ and hence $(a\odot b)\odot c=0=a\odot(b\odot c)$ proving associativity of $\odot$. Because $a\leq b'$ is equivalent to $b\leq a'$ and, moreover, $L(a,b)=L(b,a)$, $\odot$ is commutative. Further,
\begin{align*}
& \text{if }a=0\text{ then }a\odot1=0=L(0)=L(a), \\
& \text{if }a\neq0\text{ then }a\odot1=L(a,1)=L(a).
\end{align*}
The rest follows by commutativity of $\odot$. We consider the following cases:
\begin{itemize}
\item $a\leq b'$ and $b\leq c$. \\
Then $a\odot b=0\leq c$ and $a\leq1=b\rightarrow c$.
\item $a\leq b'$ and $b\not\leq c$. \\
Then $a\odot b=0\leq c$ and $a\leq b'\leq U(b',c)=b\rightarrow c$.
\item $a\not\leq b'$ and $b\leq c$. \\
Then $a\odot b=L(a,b)\leq b\leq c$ and $a\leq1=b\rightarrow c$.
\item $a\not\leq b'$, $b\not\leq c$ and $a=1$. \\
Then $a\odot b=L(a,b)=L(1,b)=L(b)\not\leq c$. Moreover, $b,c'\neq0$ and therefore $b\rightarrow c=U(b',c)=(L(b,c'))'\neq0'=1$ according to (7) whence $a=1\not\leq b\rightarrow c$.
\item $a\not\leq b'$, $b\not\leq c$ and $b=1$. \\
Then the following are equivalent:
\begin{align*}
a\odot b & \leq c, \\
  L(a,b) & \leq c, \\
  L(a,1) & \leq c, \\
    L(a) & \leq c, \\
       a & \leq c, \\
       a & \leq U(c), \\
       a & \leq U(0,c), \\
       a & \leq U(b',c), \\
       a & \leq b\rightarrow c.
\end{align*}
\item $a\not\leq b'$, $b\not\leq c$ and $c=0$. \\
Then $a,b\neq0$ and therefore $a\odot b=L(a,b)\neq0$ according to (7) whence $a\odot b\not\leq c$. Moreover, $a\not\leq U(b')=U(b',0)=U(b',c)=b\rightarrow c$.
\item $a\not\leq b'$, $b\not\leq c$, $a,b\neq1$ and $c\neq0$. \\
Then $a,b,c\neq0,1$. If $a\odot b\leq c$ then
\begin{align*}
b\rightarrow c & =U(b',c)\subseteq U(b',a\odot b)=U(b',L(a,b))=UL(U(b',a),U(b',b))= \\
               & =UL(U(b',a),U(a',a))\subseteq ULU(a)=U(a)
\end{align*}
and hence $a\leq b\rightarrow c$. If, conversely, $a\leq b\rightarrow c$ then
\begin{align*}
a\odot b & =L(a,b)\subseteq L(b\rightarrow c,b)=L(U(b',c),b)=LU(L(b',b),L(c,b))= \\
         & =LU(L(c',c),L(c,b))\subseteq LUL(c)=L(c)
\end{align*}
and hence $a\odot b\leq c$.
\end{itemize}
This shows that in any case $a\odot b\leq c$ is equivalent to $a\leq b\rightarrow c$.
\end{proof}

\begin{remark}
It seems to be impossible to generalize Theorem~\ref{th4} in such a way that the assumption of distributivity is replaced by the weaker assumption of modularity. If, e.g., $L(x,y)$ and $U(x',y)$ in {\rm(R)} are replaced by $L(U(x,y'),y)$ and $U(x',L(x,y))$, respectively, then adjointness cannot be proved in the last case considered in the proof of Theorem~\ref{th4}.
\end{remark}

As it is usual in logics satisfying the double negation law, the connectives conjunction (i.e.\ $\odot$) and implication (i.e.\ $\rightarrow$) can be derived one by the other by means of involution.

\begin{theorem}
Let $\mathbf P=(P,\leq,{}',0,1)$ be a bounded poset with an antitone involution, $\odot$ and $\rightarrow$ defined by {\rm(R)} and $a,b\in P$. Then the following hold:
\begin{enumerate}[{\rm(i)}]
\item $a\odot b=(a\rightarrow b')'$,
\item $a\rightarrow b=(a\odot b')'$.
\end{enumerate}
If, moreover, $\mathbf P$ satisfies {\rm(7)} then
\begin{enumerate}
\item[{\rm(iii)}] $a\odot b=0$ if and only if $a\leq b'$,
\item[{\rm(iv)}] $a\rightarrow b=1$ if and only if $a\leq b$.
\end{enumerate}
If, moreover, $\mathbf P$ is a strict Kleene poset {\rm(}not necessarily satisfying {\rm(7))} then
\begin{enumerate}
\item[{\rm(v)}] If $a\leq b$ and $L(a',b)=0$ then $a=b$.
\end{enumerate}
\end{theorem}

\begin{proof}
\
\begin{enumerate}[(i)]
\item
\begin{align*}
& \text{If }a\leq b'\text{ then }a\odot b=0=1'=(a\rightarrow b')', \\
& \text{if }a\not\leq b'\text{ then }a\odot b=L(a,b)=(U(a',b'))'=(a\rightarrow b')',
\end{align*}
\item According to (i) we have $a\rightarrow b=(a\rightarrow b'')''=(a\odot b')'$.
\item
\begin{align*}
& \text{If }a\leq b'\text{ then }a\odot b=0\text{ by definition}, \\
& \text{if }a\not\leq b'\text{ then }a,b\neq0\text{ and hence }a\odot b=L(a,b)\neq0\text{ according to (7)}.
\end{align*}
\item According to (ii) and (iii) the following are equivalent:
\begin{align*}
a\rightarrow b & =1, \\
  (a\odot b')' & =1, \\
     a\odot b' & =0, \\
             a & \leq b.
\end{align*}
\item
\begin{align*}
& \text{If }a=0\text{ then }L(b)=L(1,b)=L(a',b)=0\text{ and hence }a=0=b, \\
& \text{if }a=1\text{ then }b=1\text{ because of }a\leq b\text{ and hence }a=1=b, \\
& \text{if }b=0\text{ then }a=0\text{ because of }a\leq b\text{ and hence }a=1=b, \\
& \text{if }b=1\text{ then }L(a')=L(a',1)=L(a',b)=0\text{ and hence }a'=0,\text{ i.e.\ }a=1=b, \\
& \text{if }a,b\neq0,1\text{ then }U(a)=U(0,a)=U(L(a',b),a)=UL(U(a',a),U(b,a))= \\
& \quad =UL(U(b',b),U(b))=ULU(b)=U(b)\text{ which implies }a=b.
\end{align*}
\end{enumerate}
\end{proof}

\section{Twist construction}

Now we show how to construct pseudo-Kleene posets and Kleene posets from posets and distributive posets, respectively. We embed an arbitrary given poset into a pseudo-Kleene poset by using the so-called twist construction known already for distributive lattices.

For an arbitrary poset $\mathbf Q=(Q,\leq)$ and an arbitrary element $a$ of $Q$ we define
\[
P_a(\mathbf Q):=\{(x,y)\in Q^2\mid L(x,y)\leq a\leq U(x,y)\}.
\]
Let $p_1$ and $p_2$ denote the first and second projection from $P_a(\mathbf Q)$ to $Q$, respectively. In $P_a(\mathbf Q)$ we introduce a binary relation $\leq$ and a unary operation $'$ as follows:
\begin{align*}
(x,y)\leq(z,v) & :\Leftrightarrow x\leq z\text{ and }y\geq v, \\
        (x,y)' & :=(y,x).
\end{align*}
Put $\mathbb P_a(\mathbf Q):=(P_a(\mathbf Q),\leq,{}')$.

\begin{theorem}
Let $\mathbf Q=(Q,\leq)$ be a poset and $a\in Q$. Then the following hold:
\begin{enumerate}[{\rm(i)}]
\item $\mathbb P_a(\mathbf Q)$ is a pseudo-Kleene poset with fixed point $(a,a)$,
\item the mapping $x\mapsto(x,a)$ is an embedding of $\mathbf Q$ into $\mathbb P_a(\mathbf Q)$,
\item $\mathbf Q$ is distributive if and only if $\mathbb P_a(\mathbf Q)$ is a Kleene poset.
\end{enumerate}
\end{theorem}

\begin{proof}
Let $(b,c),(d,e),(f,g)\in P_a(\mathbf Q)$ and $h,i\in Q$. It is easy to show that
\begin{align*}
L(A) & =L(p_1(A))\times U(p_2(A)), \\
U(A) & =U(p_1(A))\times L(p_2(A))
\end{align*}
for all $A\subseteq P_a(\mathbf Q)$.
\begin{enumerate}[(i)]
\item Clearly, $\mathbb P_a(\mathbf Q)$ is a poset with an antitone involution and fixed point $(a,a)$. According to (i),
\[
L((b,c),(c,b))=L(b,c)\times U(b,c)\leq(a,a)\leq U(d,e)\times L(d,e)=U((d,e),(e,d)).
\]
\item We have $(h,a),(i,a)\in P_a(\mathbf Q)$. Moreover, $(h,a)\leq(i,a)$ if and only if $h\leq i$.
\item Using (i), we obtain
\begin{align*}
          L(U((b,c),(d,e)),(f,g)) & \approx L(U(b,d)\times L(c,e),(f,g))\approx \\
                                  & \approx L(U(b,d),f)\times U(L(c,e),g), \\
LU(L((b,c),(f,g)),L((d,e),(f,g))) & \approx LU(L(b,f)\times U(c,g),L(d,f)\times U(e,g))\approx \\
                                  & \approx L(U(L(b,f),L(d,f))\times L(U(c,g),U(e,g)))\approx \\
                                  & \approx LU(L(b,f),L(d,f))\times UL(U(c,g),U(e,g)).
\end{align*}
\end{enumerate}
\end{proof}

It is well-known (see e.g.\ \cite K) that for an arbitrary distributive lattice $\mathbf Q$ the constructed poset $\mathbb P_a(\mathbf Q)$ is a Kleene algebra. We have shown that this construction can be extended to arbitrary posets.

\begin{example}
If $\mathbf Q$ is the poset $(Q,\leq)$ visualized in Figure~8:
\vspace*{-2mm}
\begin{center}
\setlength{\unitlength}{7mm}
\begin{picture}(4,6)
\put(2,1){\circle*{.3}}
\put(2,3){\circle*{.3}}
\put(1,5){\circle*{.3}}
\put(3,5){\circle*{.3}}
\put(2,3){\line(0,-1)2}
\put(2,3){\line(-1,2)1}
\put(2,3){\line(1,2)1}
\put(1.85,.25){$0$}
\put(2.4,2.85){$a$}
\put(.85,5.4){$b$}
\put(2.85,5.4){$c$}
\put(1.2,-.75){{\rm Fig.\ 8}}
\end{picture}
\end{center}
\vspace*{4mm}
then $\mathbb P_a(\mathbf Q)$ is the pseudo-Kleene poset shown in Figure~9
\vspace*{-2mm}
\begin{center}
\setlength{\unitlength}{7mm}
\begin{picture}(6,10)
\put(1,1){\circle*{.3}}
\put(5,1){\circle*{.3}}
\put(1,3){\circle*{.3}}
\put(3,3){\circle*{.3}}
\put(5,3){\circle*{.3}}
\put(1,5){\circle*{.3}}
\put(3,5){\circle*{.3}}
\put(5,5){\circle*{.3}}
\put(1,7){\circle*{.3}}
\put(3,7){\circle*{.3}}
\put(5,7){\circle*{.3}}
\put(1,9){\circle*{.3}}
\put(5,9){\circle*{.3}}
\put(1,1){\line(0,1)8}
\put(1,1){\line(1,1)2}
\put(5,1){\line(-1,1)2}
\put(5,1){\line(0,1)8}
\put(3,5){\line(-1,-1)2}
\put(3,5){\line(-1,1)2}
\put(3,5){\line(0,1)2}
\put(3,5){\line(1,1)2}
\put(3,5){\line(1,-1)2}
\put(3,5){\line(0,-1)2}
\put(1,9){\line(1,-1)2}
\put(5,9){\line(-1,-1)2}
\put(.35,.25){$(0,c)$}
\put(4.35,.25){$(0,b)$}
\put(2.35,1.9){$(0,a)$}
\put(-.6,2.85){$(a,c)$}
\put(-.6,4.85){$(b,c)$}
\put(5.3,2.85){$(a,b)$}
\put(5.3,4.85){$(c,b)$}
\put(3.3,4.85){$(a,a)$}
\put(-.6,6.85){$(b,a)$}
\put(5.3,6.85){$(c,a)$}
\put(2.35,7.75){$(a,0)$}
\put(.35,9.4){$(b,0)$}
\put(4.35,9.4){$(c,0)$}
\put(2.2,-.75){{\rm Fig.\ 9}}
\end{picture}
\end{center}
\vspace*{4mm}
which is neither a pseudo-Kleene algebra nor a Kleene poset since
\begin{align*}
L(U((a,c),(0,a)),(a,b)) & =L(U((a,a)),(a,b))=L((a,a),(a,b))=\{(0,b),(a,b)\}\neq \\
                        & \neq\{(0,b)\}=L(P_a(\mathbf Q)\setminus\{(0,c),(a,c),(b,c)\})=LU((0,b))= \\
                        & =LU(L((a,c),(a,b)),L((0,a),(a,b))).
\end{align*}
\end{example}

Authors' addresses:

Ivan Chajda \\
Palack\'y University Olomouc \\
Faculty of Science \\
Department of Algebra and Geometry \\
17.\ listopadu 12 \\
771 46 Olomouc \\
Czech Republic \\
ivan.chajda@upol.cz

Helmut L\"anger \\
TU Wien \\
Faculty of Mathematics and Geoinformation \\
Institute of Discrete Mathematics and Geometry \\
Wiedner Hauptstra\ss e 8-10 \\
1040 Vienna \\
Austria, and \\
Palack\'y University Olomouc \\
Faculty of Science \\
Department of Algebra and Geometry \\
17.\ listopadu 12 \\
771 46 Olomouc \\
Czech Republic \\
helmut.laenger@tuwien.ac.at
\end{document}